\documentclass[11pt]{amsart}

\usepackage[margin=1in]{geometry}
\usepackage{amsmath,amssymb,amsthm}
\usepackage{enumitem}
\usepackage{tikz}
\usepackage[colorlinks=true,linkcolor=blue!55!black,citecolor=blue!55!black,urlcolor=blue!55!black]{hyperref}
\usepackage{microtype}

\newtheorem{theorem}{Theorem}[section]
\newtheorem{proposition}[theorem]{Proposition}
\newtheorem{lemma}[theorem]{Lemma}
\newtheorem{corollary}[theorem]{Corollary}
\theoremstyle{definition}
\newtheorem{example}[theorem]{Example}
\newtheorem{remark}[theorem]{Remark}

\newcommand{\reg}{\operatorname{reg}}
\newcommand{\m}{\mathfrak m}
\newcommand{\ZZ}{\mathbb Z}
\newcommand{\kk}{\mathbb K}
\newcommand{\Ass}{\operatorname{Ass}}
\newcommand{\cA}{\mathcal A}
\newcommand{\cB}{\mathcal B}
\newcommand{\cC}{\mathcal C}
\newcommand{\cL}{\mathcal L}
\newcommand{\cN}{\mathcal N}
\newcommand{\doi}[1]{\url{https://doi.org/#1}}

\title[Symbolic Powers and Supermodular Prime Intersections]{Regularity of Symbolic Powers of Polymatroidal Ideals and Supermodular Prime Intersections}
\author{Chwas Ahmed\textsuperscript{1,2}}
\address{\textsuperscript{1} Department of Mathematics, College of Science, University of Sulaimani, Sulaimani, Kurdistan Region, Iraq.}
\address{\textsuperscript{2} Department of Mathematics and Natural Sciences, American University of Iraq, Sulaimani, Kurdistan Region, Iraq.}
\email{chwas.ahmed@univsul.edu.iq}

\author{Mohammed Rafiq Namiq\textsuperscript{1}}
\thanks{Corresponding author: Mohammed Rafiq Namiq.}
\address{\textsuperscript{1} Department of Mathematics, College of Science, University of Sulaimani, Sulaimani, Kurdistan Region, Iraq.}
\email{mohammed.namiq@univsul.edu.iq}
\date{}


\begin{document}
	
	\begin{abstract}
		Let $S=\kk[x_1,\ldots,x_n]$ be a polynomial ring over a field $\kk$. For a polymatroidal ideal $I\subseteq S$ generated in degree $d$, we prove
		\[
		\omega(I^{(s)})=\reg I^{(s)}=\reg I^s=sd\qquad(s\ge1),
		\]
		where symbolic powers use all associated primes and $\omega$ denotes the largest degree of a minimal generator. Thus every symbolic power has a degree resolution. The argument uses a regularity bound for supermodular prime intersections, obtained from Takayama's formula and a homological estimate for noncover complexes. For every nonzero proper squarefree monomial ideal $L$, we also show $\reg L^{(s)}\le s\min\{m,\lambda(L)-\operatorname{ht}L+1\}$, where $m$ counts the minimal primes and $\lambda(L)$ is the support size, and obtain equality for skeleton and pure skeleton ideals. Finally, for every $d\ge3$, we construct a simple graphic matroid whose basis ideal is generated in degree $d$ and has a second symbolic power that is not componentwise linear.
	\end{abstract}
	
	\maketitle
	
	\section{Introduction}
	
	Symbolic powers can behave differently from ordinary powers, even for monomial ideals. Polymatroidal ideals are a useful class for studying this difference because their ordinary powers have a particularly simple homological behavior. Let $\kk$ be a field and let $S=\kk[x_1,\ldots,x_n]$ be the standard graded polynomial ring. Write $\mathcal{G}(J)$ for the minimal monomial generating set of a monomial ideal $J$. A monomial ideal $I\subseteq S$ generated in one degree is \emph{polymatroidal} if its minimal generators satisfy the exchange property. Namely, whenever $x^{\mathbf a},x^{\mathbf b}\in \mathcal{G}(I)$ and $a_i>b_i$ for some $i$, there is an index $j$ with $a_j<b_j$ such that $x_jx^{\mathbf a}/x_i\in \mathcal{G}(I)$. Polymatroidal ideals have linear quotients, and a product of polymatroidal ideals is again polymatroidal \cite{ConcaHerzog,HerzogHibi}. If $I$ is generated in degree $d$, it follows that $I^s$ has an $sd$-linear resolution and $\reg I^s=sd$ for every $s\ge1$. We ask whether the symbolic powers of $I$ have the same regularity.
	
	Throughout, we use symbolic powers defined by all associated primes as follows.
	\begin{equation}\label{eq:1}
		I^{(s)}=\bigcap_{P\in\Ass(S/I)}\bigl(I^sS_P\cap S\bigr),\qquad s\ge1.
	\end{equation}
	This convention is important because a polymatroidal ideal can have embedded associated primes. For example, let $\m=(x_1,x_2,x_3)$ and $I=(x_1^2,x_1x_2,x_1x_3)=x_1\m\subseteq\kk[x_1,x_2,x_3]$. Then $I$ is polymatroidal, and $I=(x_1)\cap\m^2$ is an irredundant primary decomposition. We have $\Ass(S/I)=\{(x_1),\m\}$. The prime $(x_1)$ is minimal and $\m$ is embedded. If one used only the minimal prime, then already for $s=1$ one would have $IS_{(x_1)}\cap S=(x_1)\ne I$, because $x_2$ and $x_3$ are units in $S_{(x_1)}$. So the embedded prime is essential in \eqref{eq:1}. Proposition~\ref{prop:3} describes how all associated primes enter the prime intersection for $I^{(s)}$.
	
	Ficarra and Moradi \cite{FicarraMoradi} conjectured that $\reg I^{(s)}=\reg I^s$ for every polymatroidal ideal $I$ and every $s\ge1$ (Conjecture~A), and that every $I^{(s)}$ is componentwise linear (Conjecture~B). They proved the regularity lower bound when $I$ has no embedded associated primes and verified the two conjectures for several classes. Theorem~\ref{thm:1} proves Conjecture~A without assuming that $I$ has no embedded associated primes.
	
	For a nonzero monomial ideal $J$, write $\omega(J)=\max\{\deg u:u\in \mathcal{G}(J)\}$. Following \cite{Namiq}, we say that $J$ has a \emph{degree resolution} if $\reg J=\omega(J)$.
	
	\begin{theorem}\label{thm:1}
		Let $I\subseteq S$ be a polymatroidal ideal generated in degree $d\ge1$. Then $I^{(s)}$ has a degree resolution for every $s\ge1$. More precisely,
		\[
		\omega(I^{(s)})=\reg I^{(s)}=\reg I^s=sd.
		\]
	\end{theorem}
	
	Theorem~\ref{thm:1} follows from a more general prime intersection construction. Set $E=[n]$ and $P_F=(x_i:i\in F)$ for $F\subseteq E$. Let $\cA\subseteq2^E$ be a lower ideal, and let $\tau:2^E\to\ZZ_{\ge0}$ be supermodular with $\tau(\emptyset)=0$. For
	\[
	J_s(\cA,\tau)=\bigcap_{F\in\cA}P_F^{s\tau(F)},
	\]
	Theorem~\ref{thm:3} proves
	\[
	s\omega(J_1)\le\omega(J_s)\le\reg J_s\le s\tau(E),
	\]
	and, in addition, $x^{\mathbf u}\in \mathcal{G}(J_1)$ implies $x^{s\mathbf u}\in \mathcal{G}(J_s)$. The scaling property for minimal generators gives the lower bound, while Takayama's formula together with Theorem~\ref{thm:2} gives the regularity bound. If $\omega(J_1)=\tau(E)$, all inequalities are equalities, so every $J_s$ has a degree resolution and $\omega(J_s)=\reg J_s=s\tau(E)$ for all $s\ge1$.
	
	For a polymatroidal ideal, the function $\tau$ is defined in terms of the rank function of the corresponding discrete polymatroid. If $\rho$ is the rank function and $d=\rho(E)$, then $\tau(F)=d-\rho(E\setminus F)$ is supermodular. Proposition~\ref{prop:3} gives the representation $I^{(s)}=J_s(\cA(I),\tau)$, where $\cA(I)$ is determined by all primes in $\Ass(S/I)$. Since $\tau(E)=d$ and $I=J_1$ is generated in degree $d$, Theorem~\ref{thm:1} follows from the equality case of Theorem~\ref{thm:3}. 
	This is why all associated primes in the definition of symbolic powers \eqref{eq:1} are needed in the proof.
	
	The scaling property for minimal generators is related to results of Ficarra and Moradi \cite[Theorem~1.2]{FicarraMoradi} and Dung, Hien, Nguyen, and Trung \cite[Lemma~4.2(ii)]{DungHienNguyenTrung}, but the hypotheses are different. Here the ideals $J_s$ are required to have a common prime intersection representation in which the exponents are given by a single supermodular function. Under this hypothesis, the same intersection representation gives both the scaling property for minimal generators and the regularity bound. Theorem~\ref{thm:3} also applies beyond the class of polymatroidal ideals; see Example~\ref{ex:1}. Its hypotheses are therefore different from, rather than stronger than, the hypotheses in the cited scaling results.
	
	We also consider squarefree monomial ideals. For a nonzero squarefree monomial ideal $L$, set $\operatorname{supp}(L)=\{i\in[n]:x_i\text{ divides some }u\in \mathcal{G}(L)\}$, $\lambda(L)=\deg\operatorname{lcm}(\mathcal{G}(L))$, $m=|\operatorname{Min}(S/L)|$, and $D=\lambda(L)-\operatorname{ht}L+1$. Proposition~\ref{prop:2} proves the bound
	\begin{equation*}
		\omega(L^{(s)})\le sD.
	\end{equation*}
	Corollary~\ref{cor:2} strengthens this to
	\[
	s\omega(L)\le\omega(L^{(s)})\le\reg L^{(s)}\le s\min\{m,D\}.
	\]
	In particular, $\omega(L)\le\min\{m,D\}$. If equality holds, Corollary~\ref{cor:3} implies that every $L^{(s)}$ has a degree resolution and $\omega(L^{(s)})=\reg L^{(s)}=s\omega(L)$ for all $s\ge1$. This applies to every Stanley--Reisner ideal $I_\Gamma$ with $\dim(S/I_\Gamma)=\omega(I_\Gamma)-1$, including the skeleton and pure skeleton ideals in Corollary~\ref{cor:4}. For unmixed squarefree ideals, the $sD$ bound also follows directly from the supermodular construction. See Remark~\ref{rem:3}.
	
	Having a degree resolution does not imply componentwise linearity. Proposition~\ref{prop:4} constructs, for every $d\ge3$, a simple graphic matroid $M_d$ of rank $d$ such that $I_B(M_d)^{(2)}$ is not componentwise linear. Ficarra and Moradi proved that Conjectures~A and~B hold for every polymatroidal ideal generated in degree $2$ \cite[Theorem~5.8(a) and Proposition~6.1]{FicarraMoradi}. Together with Proposition~\ref{prop:4}, this proves the criterion in Corollary~\ref{cor:8}, namely that every symbolic power of every polymatroidal ideal generated in degree $d$ is componentwise linear if and only if $d\le2$. The rank $3$ example is the cycle matroid of $K_4$. Higher rank examples are formed by adjoining coloops. The same example gives a negative answer to the componentwise linearity part of Question~D in \cite{FicarraMoradi}. Related counterexamples for higher symbolic powers of co-chordal edge ideals are given in \cite{AhmedNamiqCoChordal}; they concern a different class of ideals. See also \cite{FicarraMoradiRomer} for componentwise linear symbolic powers of edge ideals.
	
	The rest of the paper is divided into three sections. Section~\ref{sec:1} proves Theorem~\ref{thm:2}. Section~\ref{sec:2} applies it to Takayama degree complexes, proves Theorem~\ref{thm:3}, and derives the squarefree consequences in Proposition~\ref{prop:2} and Corollaries~\ref{cor:2}--\ref{cor:4}. Section~\ref{sec:3} develops the discrete polymatroid representation, proves Proposition~\ref{prop:3} and Theorem~\ref{thm:1}, records further consequences in Corollaries~\ref{cor:5}--\ref{cor:7}, and ends with the componentwise linearity results in Proposition~\ref{prop:4} and Corollary~\ref{cor:8}.
	
	\section{Supermodular functions and noncover complexes}\label{sec:1}
	
	Let $E$ be a finite set. A family $\cA\subseteq2^E$ is a \emph{lower ideal} if $F\in\cA$ and $G\subseteq F$ then $G\in\cA$. Let $\delta:2^E\longrightarrow\ZZ$ be an integer-valued supermodular function with $\delta(\emptyset)=0$, meaning that
	\begin{equation*}
		\delta(A)+\delta(B)\le\delta(A\cap B)+\delta(A\cup B)
		\qquad(A,B\subseteq E).
	\end{equation*}
	Only the values of $\delta$ on $\cA$ are needed below. Define the family of minimal positive sets
	\[
	\cC_\delta=
	\bigl\{F\in\cA:\delta(F)>0\text{ and }\delta(G)\le0
	\text{ for every }G\subsetneq F\bigr\}.
	\]
	The associated \emph{noncover complex} is
	\begin{equation}\label{eq:2}
		\Delta_\delta=\left\langle E\setminus F:F\in\cC_\delta\right\rangle.
	\end{equation}
	The complex $\Delta_\delta$ is the usual noncover complex of the hypergraph with edge set $\cC_\delta$. Noncover complexes of graphs and hypergraphs have been studied, for example, in \cite{ChoiKimPark,KimKim}. If $\cC_\delta=\emptyset$, we take $\Delta_\delta=\emptyset$ to be the void complex. Here the role of the supermodular function $\delta$ is to provide the numerical growth condition needed for the homology bound below.
	
	\begin{example}
		Let $E=\{1,2,3\}$, $\cA=2^E$, and $\delta(F)=|F\cap\{1\}|$. Then $\delta$ is supermodular, and $\cC_\delta=\{\{1\}\}$. Therefore $\Delta_\delta=\langle\{2,3\}\rangle$, the simplex on vertices $2$ and $3$.
	\end{example}
	
	The next theorem is the combinatorial estimate needed later in the regularity argument.
	
	\begin{theorem}\label{thm:2}
		For every $q\ge-1$ and every field $\kk$, if $\widetilde H_q(\Delta_\delta;\kk)\ne0$, then
		\[
		\delta(E)\ge q+2.
		\]
	\end{theorem}
	
	We use three facts in the proof. The positive sets define the same noncover complex as $\cC_\delta$. This complex is homotopy equivalent to the order complex of the proper nonempty unions of members of $\cC_\delta$. Finally, supermodularity and integrality imply that $\delta$ increases by at least one along each strict inclusion between such unions.
	
	\begin{lemma}\label{lem:1}
		The noncover complex, defined in \eqref{eq:2}, satisfies
		\[
		\Delta_\delta=
		\left\langle E\setminus F:F\in\cA,\ \delta(F)>0\right\rangle.
		\]
	\end{lemma}
	
	\begin{proof}
		Set $\Sigma=\langle E\setminus F:F\in\cA,\ \delta(F)>0\rangle$. Every member of $\cC_\delta$ is positive, so $\Delta_\delta\subseteq\Sigma$. Conversely, take $F\in\cA$ with $\delta(F)>0$. Since $E$ is finite, $F$ contains a minimal positive subset $G$. Because $\cA$ is a lower ideal, $G\in\cA$. By the minimality of $G$, we have $G\in\cC_\delta$. Since $G\subseteq F$, we have $E\setminus F\subseteq E\setminus G$. It follows that the simplex on $E\setminus F$ is a face of a generating simplex of $\Delta_\delta$. Every generator of $\Sigma$ is therefore a face of $\Delta_\delta$, and $\Sigma=\Delta_\delta$.
	\end{proof}
	
	If $q\ge0$ and $E\in\cC_\delta$, then the corresponding generating simplex is empty. We therefore write $\cC_\delta^\circ=\{F\in\cC_\delta:F\subsetneq E\}$.
	
	\begin{lemma}\label{lem:2}
		Assume $\widetilde H_q(\Delta_\delta;\kk)\ne0$ for some $q\ge0$. Then the members of $\cC_\delta^\circ$ cover $E$. Define
		\[
		\cN_\delta=
		\left\{\cB\subseteq\cC_\delta^\circ:
		\bigcup_{F\in\cB}F\ne E\right\},
		\]
		and let
		\[
		\cL_\delta=
		\left\{U:\emptyset\ne U\subsetneq E,\
		U=\bigcup_{F\in\cB}F
		\text{ for some nonempty }\cB\subseteq\cC_\delta^\circ\right\},
		\]
		ordered by inclusion. Then
		\[
		\Delta_\delta\simeq\cN_\delta\simeq\Delta(\cL_\delta).
		\]
	\end{lemma}
	
	\begin{proof}
		Suppose first that some $e\in E$ lies in no member of $\cC_\delta^\circ$. Since $q\ge0$ and $\widetilde H_q(\Delta_\delta;\kk)\ne0$, the complex has a nonempty face. Every nonempty generating simplex has the form $E\setminus F$ with $F\in\cC_\delta^\circ$, and each such simplex contains $e$. It follows that every face lies in a simplex containing $e$, so $\Delta_\delta$ is a cone with apex $e$. A cone has zero reduced homology, a contradiction. We conclude that $\bigcup_{F\in\cC_\delta^\circ}F=E$.
		
		Consider the cover of $|\Delta_\delta|$ by the nonempty simplices $E\setminus F$, $F\in\cC_\delta^\circ$. For a subfamily $\cB\subseteq\cC_\delta^\circ$, the corresponding intersection is $\bigcap_{F\in\cB}(E\setminus F)=E\setminus\bigcup_{F\in\cB}F$. It is nonempty if and only if $\bigcup_{F\in\cB}F\ne E$, which means exactly that $\cB$ is a face of $\cN_\delta$. Every nonempty intersection is itself a simplex and hence contractible. The nerve lemma therefore gives $\Delta_\delta\simeq\cN_\delta$.
		
		Finally, let $\mathcal P(\cN_\delta)$ be the poset of nonempty faces of $\cN_\delta$ and define
		\[
		\phi:\mathcal P(\cN_\delta)\longrightarrow\cL_\delta,
		\qquad
		\phi(\cB)=\bigcup_{F\in\cB}F.
		\]
		Fix $U\in\cL_\delta$ and set $\cC_U=\{F\in\cC_\delta^\circ:F\subseteq U\}$. By the definition of $\cL_\delta$, some subfamily has union $U$, so the union of all members of $\cC_U$ is also $U$. Since $U\subsetneq E$, $\cC_U$ is a face of $\cN_\delta$. Moreover, if $\phi(\cB)\subseteq U$, then every $F\in\cB$ is contained in $U$, so $\cB\subseteq\cC_U$. Thus $\cC_U$ is the maximum element of the lower fiber $\phi^{-1}((\cL_\delta)_{\le U})$. Every such fiber is contractible. The poset fiber theorem \cite{Bjorner} gives
		\[
		\Delta(\mathcal P(\cN_\delta))\simeq\Delta(\cL_\delta).
		\]
		Since $\Delta(\mathcal P(\cN_\delta))$ is the barycentric subdivision of $\cN_\delta$, it follows that $\Delta_\delta\simeq\cN_\delta\simeq\Delta(\cL_\delta)$.
	\end{proof}
	
	\begin{lemma}\label{lem:3}
		If $U\subsetneq V$ are unions of members of $\cC_\delta$, allowing the empty union, then
		\[
		\delta(V)\ge\delta(U)+1.
		\]
	\end{lemma}
	
	\begin{proof}
		Write $V=F_1\cup\cdots\cup F_r$ with $F_j\in\cC_\delta$. Starting with $U$, add the sets $F_j$ one at a time and omit repetitions. This produces a strict chain
		\[
		U=U_0\subsetneq U_1\subsetneq\cdots\subsetneq U_m=V.
		\]
		It is enough to prove that $\delta(U_j)\ge\delta(U_{j-1})+1$ at each step.
		
		Fix $j$. Then $U_j=U_{j-1}\cup F$ for some $F\in\cC_\delta$ with $F\nsubseteq U_{j-1}$. It follows that $U_{j-1}\cap F\subsetneq F$. Since $F\in\cA$ and $\cA$ is a lower ideal, $U_{j-1}\cap F\in\cA$. By the minimality of $F$ among the positive sets, $\delta(U_{j-1}\cap F)\le0$. Also $\delta(F)>0$, and integrality strengthens this to $\delta(F)\ge1$. Applying supermodularity to $U_{j-1}$ and $F$ gives
		\[
		\delta(U_j)+\delta(U_{j-1}\cap F)
		\ge\delta(U_{j-1})+\delta(F).
		\]
		Using the two inequalities above gives
		\[
		\delta(U_j)\ge\delta(U_{j-1})+1.
		\]
		Iterating this inequality along the chain gives $\delta(V)\ge\delta(U)+m$, and in particular $\delta(V)\ge\delta(U)+1$.
	\end{proof}
	
	\begin{proof}[Proof of Theorem~\ref{thm:2}]
		Assume first that $q\ge0$. By Lemma~\ref{lem:2}, $\Delta_\delta$ is homotopy equivalent to $\Delta(\cL_\delta)$. Let
		\[
		U_1\subsetneq\cdots\subsetneq U_t
		\]
		be a strict chain in $\cL_\delta$. Lemma~\ref{lem:2} shows that $E$ is a union of members of $\cC_\delta$, while $\emptyset$ is the empty union. Therefore Lemma~\ref{lem:3} applies to every inclusion in
		\[
		\emptyset\subsetneq U_1\subsetneq\cdots\subsetneq U_t\subsetneq E.
		\]
		There are $t+1$ strict inclusions. Since $\delta(\emptyset)=0$, repeated application of Lemma~\ref{lem:3} gives
		\[
		\delta(E)\ge t+1.
		\]
		The chain $U_1\subsetneq\cdots\subsetneq U_t$ determines a simplex of $\Delta(\cL_\delta)$ of dimension $t-1$. Every simplex therefore has dimension at most $\delta(E)-2$, and hence
		\[
		\dim\Delta(\cL_\delta)\le\delta(E)-2.
		\]
		If $\widetilde H_q(\Delta_\delta;\kk)\ne0$, then $\widetilde H_q(\Delta(\cL_\delta);\kk)\ne0$ as well, and hence $q\le\dim\Delta(\cL_\delta)$. This gives $\delta(E)\ge q+2$.
		
		Now let $q=-1$. Nonzero $\widetilde H_{-1}(\Delta_\delta;\kk)$ means $\Delta_\delta=\{\emptyset\}$. The complex is not void, so the generating family in \eqref{eq:2} contains at least one simplex $E\setminus F$. Since the only face is empty, $E\setminus F=\emptyset$, hence $F=E\in\cC_\delta$. Since $\delta(E)>0$, integrality gives $\delta(E)\ge1=q+2$.
	\end{proof}
	
	\begin{remark}
		\begin{enumerate}[label=\textup{(\roman*)}]
			\item\label{rem:1} Integrality is essential. Let $E=\{1,2\}$, $\cA=2^E$, and set $\delta(\emptyset)=0$, $\delta(\{1\})=\delta(\{2\})=1/10$, and $\delta(E)=1/5$. Then $\delta$ is supermodular and $\Delta_\delta$ consists of two isolated vertices. Thus $\widetilde H_0(\Delta_\delta;\kk)\ne0$, although $\delta(E)<2$. The proof of Lemma~\ref{lem:3} uses integrality precisely when $\delta(F)>0$ is replaced by $\delta(F)\ge1$.
			\item\label{rem:2} The bound is sharp in every homological degree. Fix $q\ge-1$, let $|E|=q+2$, take $\cA=2^E$, and set $\delta(F)=|F|$. Then the minimal positive sets are the singletons, $\Delta_\delta$ is the boundary of the $(q+1)$-simplex, and $\delta(E)=q+2$.
			\item The argument is independent of the characteristic of $\kk$. The homology group may depend on the field, but whenever it is nonzero the same numerical bound follows.
		\end{enumerate}
	\end{remark}
	
	\section{Supermodular prime intersections and regularity}\label{sec:2}
	
	Let $E=[n]$, let $\cA\subseteq2^E$ be a lower ideal, and let $\tau:2^E\to\ZZ_{\ge0}$ be supermodular with $\tau(\emptyset)=0$. Variables outside $U=\bigcup_{F\in\cA}F$ do not occur in any prime $P_F$ with $F\in\cA$. The ring $S$ is a polynomial extension of the polynomial ring on the variables indexed by $U$, and each ideal considered below is extended from that smaller ring. Deleting the variables outside $U$ does not change either the degrees of the minimal generators or the regularity. After restricting $\tau$ to $2^U$, we may replace $E$ by $U$ and assume $E=\bigcup_{F\in\cA}F$. The function $\tau$ is nondecreasing. To see this, if $e\notin A$, supermodularity gives $\tau(A)+\tau(\{e\})\le\tau(A\cup\{e\})+\tau(\emptyset)$. Since $\tau(\emptyset)=0$ and $\tau(\{e\})\ge0$, we have $\tau(A)\le\tau(A\cup\{e\})$. Adding the elements of $B\setminus A$ one at a time gives $\tau(A)\le\tau(B)$ whenever $A\subseteq B$.
	
	For $s\ge1$, set
	\begin{equation}\label{eq:3}
		J_s=J_s(\cA,\tau)=\bigcap_{F\in\cA}P_F^{s\tau(F)},
		\qquad P_F=(x_i:i\in F),
	\end{equation}
	where factors of exponent zero are omitted.
	
	We recall the part of Takayama's formula needed below. Let $J\subseteq S$ be a monomial ideal and $\mathbf a=(a_1,\ldots,a_n)\in\ZZ^n$. Put $N=N(\mathbf a)=\{i:a_i<0\}$ and $S_U=S[x_i^{-1}:i\in U]$. The Takayama degree complex in multidegree $\mathbf a$ is
	\[
	\Delta_{\mathbf a}(J)=
	\{A\subseteq E\setminus N:x^{\mathbf a}\notin JS_{A\cup N}\}.
	\]
	With the standard reduced homology conventions, Takayama's formula \cite[Theorem~1]{Takayama} is
	\begin{equation}\label{eq:4}
		\dim_\kk H^\ell_{\m}(S/J)_{\mathbf a}
		=
		\dim_\kk\widetilde H_{\ell-|N|-1}
		(\Delta_{\mathbf a}(J);\kk).
	\end{equation}
	We also use the following formula from \cite[Proposition~2.5(ii)]{CHHKTT}.
	\[
	\reg(S/J)=
	\max\{ |\mathbf a|+\ell:
	H^\ell_{\m}(S/J)_{\mathbf a}\ne0\}.
	\]
	
	The degree complex depends on the set of negative coordinates of $\mathbf a$, not on their actual negative values. For an upper bound on regularity, we may therefore replace each negative coordinate by $-1$. The degree complex is unchanged and $|\mathbf a|$ can only increase. We call such a multidegree \emph{normalized}. Put $V=E\setminus N$ and, for $F\subseteq V$, write $a(F)=\sum_{i\in F}a_i$. Let $\mathbf e_i$ denote the $i$th standard unit vector. All $a_i$ with $i\in V$ are nonnegative. If $\widetilde H_q(\Delta_{\mathbf a}(J);\kk)\ne0$, then Takayama's formula \eqref{eq:4} gives $\ell=q+|N|+1$. Since $|\mathbf a|=a(V)-|N|$ and $\reg J=\reg(S/J)+1$, the value entering the regularity formula is
	\begin{equation*}
		|\mathbf a|+\ell+1=a(V)+q+2.
	\end{equation*}
	It is therefore enough to bound the expression on the right whenever the degree complex has nonzero reduced homology.
	
	\begin{proposition}\label{prop:1}
		Let $\mathbf a\in\ZZ^E$ be normalized, put $N=N(\mathbf a)$ and $V=E\setminus N$. Then
		\[
		\Delta_{\mathbf a}(J_s)=
		\left\langle V\setminus F:
		F\in\cA,\ F\subseteq V,\ a(F)<s\tau(F)\right\rangle.
		\]
	\end{proposition}
	
	\begin{proof}
		Fix $A\subseteq V$. By definition, $A\in\Delta_{\mathbf a}(J_s)$ if and only if $x^{\mathbf a}\notin J_sS_{A\cup N}$. Because the defining intersection for $J_s$ in \eqref{eq:3} is finite, localization commutes with it. It follows that the condition on the right holds if and only if $x^{\mathbf a}\notin P_F^{s\tau(F)}S_{A\cup N}$ for at least one factor. Only certain factors can affect this condition.
		
		If $F\cap(A\cup N)\ne\emptyset$, choose $i$ in the intersection. Then $x_i$ is both a generator of $P_F$ and a unit in $S_{A\cup N}$, so $P_F^{s\tau(F)}S_{A\cup N}=S_{A\cup N}$. Such a factor does not affect membership in the localized intersection.
		
		Suppose instead that $F\cap(A\cup N)=\emptyset$. Since $V=E\setminus N$, this is equivalent to $F\subseteq V$ and $A\subseteq V\setminus F$. No variable indexed by $F$ is inverted. Membership in the power of the monomial prime is therefore determined by the total exponent on $F$: $x^{\mathbf a}\in P_F^{s\tau(F)}S_{A\cup N}$ if and only if $a(F)\ge s\tau(F)$. So $A$ is a face of the degree complex if and only if it is contained in $V\setminus F$ for some $F\in\cA$, with $F\subseteq V$ and $a(F)<s\tau(F)$. This proves Proposition~\ref{prop:1}.
	\end{proof}
	
	\begin{theorem}\label{thm:3}
		Assume that $J_1$ is proper. Then, for every $s\ge1$,
		\[
		s\omega(J_1)\le\omega(J_s)\le\reg J_s\le s\tau(E).
		\]
		Moreover, if $x^{\mathbf u}\in \mathcal{G}(J_1)$, then $x^{s\mathbf u}\in \mathcal{G}(J_s)$ for every $s\ge1$. If $\omega(J_1)=\tau(E)$, then every $J_s$ has a degree resolution and
		\[
		\omega(J_s)=\reg J_s=s\tau(E)
		\qquad(s\ge1).
		\]
	\end{theorem}
	
	\begin{proof}
		We first prove the regularity bound. Fix $s\ge1$ and a normalized multidegree $\mathbf a$ such that $\widetilde H_q(\Delta_{\mathbf a}(J_s);\kk)\ne0$. Let $N=N(\mathbf a)$ and $V=E\setminus N$. The case $V=\emptyset$ cannot occur. In that case every variable is inverted in $S_N$, every nonzero monomial ideal becomes the whole ring, and the degree complex is void, whose reduced homology is zero.
		
		For $F\subseteq V$, define $\delta(F)=s\tau(F)-a(F)$. The function $F\mapsto a(F)$ is modular. Thus $\delta$ is integer-valued and supermodular, with $\delta(\emptyset)=0$.
		
		Restrict $\cA$ to $2^V$. Proposition~\ref{prop:1} identifies $\Delta_{\mathbf a}(J_s)$ with the complex generated by $V\setminus F$ for those $F$ with $\delta(F)>0$. By Lemma~\ref{lem:1}, this is the noncover complex associated with the minimal positive sets of $\delta$. Theorem~\ref{thm:2}, now applied with ground set $V$, gives
		\[
		s\tau(V)-a(V)=\delta(V)\ge q+2.
		\]
		Equivalently,
		\[
		a(V)+q+2\le s\tau(V)\le s\tau(E),
		\]
		where the last inequality uses the monotonicity of $\tau$. By the regularity contribution formula above, the multidegree $\mathbf a$ contributes the value $a(V)+q+2$ to the regularity formula. The same bound holds for every multidegree with nonzero homology, so
		\[
		\reg J_s\le s\tau(E).
		\]
		
		It remains to prove the assertion on minimal generators. Let $x^{\mathbf u}\in \mathcal{G}(J_1)$. Membership in $J_1$ is equivalent to the system
		\begin{equation*}
			u(F)\ge\tau(F)
			\qquad(F\in\cA),
		\end{equation*}
		where $u(F)=\sum_{i\in F}u_i$.
		
		Fix $i$ with $u_i>0$. Minimality of $x^{\mathbf u}$ means $x^{\mathbf u-\mathbf e_i}\notin J_1$. At least one of these inequalities therefore fails after decreasing the $i$th coordinate. Choose $F_i\in\cA$ such that $(\mathbf u-\mathbf e_i)(F_i)<\tau(F_i)$. Necessarily $i\in F_i$. Otherwise the left side would equal $u(F_i)\ge\tau(F_i)$. We obtain $u(F_i)-1<\tau(F_i)\le u(F_i)$. All quantities are integers, so
		\begin{equation*}
			u(F_i)=\tau(F_i).
		\end{equation*}
		For every variable dividing $x^{\mathbf u}$, we have therefore found a defining inequality that is an equality.
		
		Multiplying the defining inequalities by $s$ shows $x^{s\mathbf u}\in J_s$. To check minimality, divide by any variable $x_i$ with $u_i>0$. The equality $u(F_i)=\tau(F_i)$ gives $(s\mathbf u-\mathbf e_i)(F_i)=s\tau(F_i)-1<s\tau(F_i)$, so $x^{s\mathbf u-\mathbf e_i}\notin J_s$. Therefore $x^{s\mathbf u}/x_i\notin J_s$ for every variable $x_i$ dividing $x^{s\mathbf u}$. This is the usual minimality criterion for a monomial generator. Therefore $x^{s\mathbf u}\in \mathcal{G}(J_s)$. Choose $x^{\mathbf u}\in \mathcal{G}(J_1)$ with $|\mathbf u|=\omega(J_1)$. Its scaled generator has degree $s\omega(J_1)$, so $s\omega(J_1)\le\omega(J_s)$. Together with the regularity bound and the general inequality $\omega(J_s)\le\reg J_s$, this gives
		\[
		s\omega(J_1)\le\omega(J_s)\le\reg J_s\le s\tau(E).
		\]
		If $\omega(J_1)=\tau(E)$, the first and last terms coincide, so every inequality is an equality. We then have $\reg J_s=\omega(J_s)$ for every $s\ge1$, and each $J_s$ has a degree resolution. This proves the final assertion.
	\end{proof}
	
	\begin{corollary}\label{cor:1}
		Let $L\subsetneq S$ be a nonzero monomial ideal. Suppose there are a lower ideal $\cA\subseteq2^E$ and a supermodular function $\tau:2^E\to\ZZ_{\ge0}$ with $\tau(\emptyset)=0$ such that
		\[
		L^{(s)}=\bigcap_{F\in\cA}P_F^{s\tau(F)}
		\qquad(s\ge1).
		\]
		Then
		\[
		s\omega(L)\le\omega(L^{(s)})\le\reg L^{(s)}\le s\tau(E)
		\qquad(s\ge1).
		\]
		If $\omega(L)=\tau(E)$, then every $L^{(s)}$ has a degree resolution and
		\[
		\omega(L^{(s)})=\reg L^{(s)}=s\tau(E).
		\]
	\end{corollary}
	
	\begin{proof}
		By the hypothesis of Corollary~\ref{cor:1}, $L^{(s)}=J_s(\cA,\tau)$ for every $s$, and $L=J_1$. Apply Theorem~\ref{thm:3}.
	\end{proof}
	
	\begin{example}\label{ex:1}
		Let $L=(x_1,x_2)\cap(x_2,x_3)=(x_2,x_1x_3)\subseteq\kk[x_1,x_2,x_3]$. This squarefree radical ideal is not polymatroidal because its minimal generators have different degrees. Let $\cA=2^{\{1,2\}}\cup2^{\{2,3\}}$ and define $\tau(F)=\max\{|F|-1,0\}$ on $2^{\{1,2,3\}}$. The first differences of $k\mapsto\max\{k-1,0\}$ are nondecreasing, so the associated cardinality function is supermodular. On $\cA$, the only sets with positive $\tau$ are $\{1,2\}$ and $\{2,3\}$, both with value $1$. Since $L$ is radical,
		\[
		L^{(s)}=(x_1,x_2)^s\cap(x_2,x_3)^s=J_s(\cA,\tau).
		\]
		Here $\omega(L)=2=\tau(E)$, so Corollary~\ref{cor:1} shows that every $L^{(s)}$ has a degree resolution and
		\[
		\omega(L^{(s)})=\reg L^{(s)}=2s
		\qquad(s\ge1).
		\]
	\end{example}
	
	We next consider squarefree monomial ideals. H\`a and Trung showed that $\omega(L^{(s)})\le s\omega(L)$ need not hold in general \cite[Section~6]{HaNgo}. The next proposition replaces $\omega(L)$ by a bound determined by the support and the height. For a simplicial complex $\Gamma$, let $\mathcal F(\Gamma)$ denote its facets and let $I_\Gamma$ denote its Stanley--Reisner ideal.
	
	\begin{proposition}\label{prop:2}
		Let $L\subsetneq S$ be a nonzero squarefree monomial ideal and let $h=\operatorname{ht}L$. Then, for every $s\ge1$,
		\[
		\omega(L^{(s)})\le s\bigl(\lambda(L)-h+1\bigr).
		\]
		Equivalently, if $T$ is the polynomial ring on the variables that occur in $L$, then
		\[
		\omega(L^{(s)})\le s\bigl(\dim(T/L)+1\bigr).
		\]
	\end{proposition}
	
	\begin{proof}
		Let $E=\operatorname{supp}(L)$ and work in $T=\kk[x_i:i\in E]$. Because $L$ is squarefree, $|E|=\lambda(L)$ and $\operatorname{ht}_T L=h$. Write $L=I_\Gamma$. Since $L$ is radical,
		\[
		L^{(s)}=\bigcap_{F\in\mathcal F(\Gamma)}P_{E\setminus F}^{\,s}.
		\]
		Put $d=\dim(T/L)=|E|-h$. Every facet $F$ of $\Gamma$ has $|F|\le d$.
		
		Choose $x^{\mathbf a}\in \mathcal{G}(L^{(s)})$ and fix $j$ with $a_j>0$. Since $x^{\mathbf a-\mathbf e_j}\notin L^{(s)}$, there is a facet $F_j$ such that
		\[
		(\mathbf a-\mathbf e_j)(E\setminus F_j)<s\le a(E\setminus F_j).
		\]
		Necessarily $j\notin F_j$, and all quantities are integers. Hence $a(E\setminus F_j)=s$. In particular,
		\begin{equation*}
			a_j\le s
			\qquad\text{for every }j\in\operatorname{supp}(\mathbf a).
		\end{equation*}
		
		Fix one such $j$. Then
		\[
		|\mathbf a|=a(E\setminus F_j)+a(F_j)
		=s+a(F_j)\le s+s|F_j|\le s(d+1).
		\]
		Therefore every minimal generator of $L^{(s)}$ has degree at most $s(d+1)=s\bigl(\lambda(L)-h+1\bigr)$. Since $d=\dim(T/L)$, the equivalent form follows as well.
	\end{proof}
	
	For $s=1$, Proposition~\ref{prop:2} proves the bound $\omega(L)\le D$. The next corollary shows directly that $D$ also controls the regularity of every symbolic power and combines this estimate with a bound in terms of the number of minimal primes.
	
	\begin{corollary}\label{cor:2}
		Let $L\subsetneq S$ be a nonzero squarefree monomial ideal, let $T$ be the polynomial ring on the variables that occur in $L$, and put $m=|\operatorname{Min}(S/L)|$ and $D=\lambda(L)-\operatorname{ht}L+1=\dim(T/L)+1$. Then, for every $s\ge1$,
		\[
		s\omega(L)\le\omega(L^{(s)})\le\reg L^{(s)}\le s\min\{m,D\}.
		\]
		In particular, $\omega(L)\le\min\{m,D\}$.
	\end{corollary}
	
	\begin{proof}
		Write $\operatorname{Min}(S/L)=\{P_{C_1},\ldots,P_{C_m}\}$. Since $L$ is squarefree, $L^{(s)}=\bigcap_{j=1}^mP_{C_j}^s$. The sets $C_1,\ldots,C_m$ form an antichain. Let
		\[
		\cA=\bigcup_{j=1}^m2^{C_j}
		\qquad\text{and}\qquad
		\tau(F)=\left|\{j:C_j\subseteq F\}\right|.
		\]
		For a fixed $C\subseteq E$, the indicator $\varepsilon_C(F)=1$ if $C\subseteq F$ and $0$ otherwise is supermodular. Hence $\tau=\sum_{j=1}^m\varepsilon_{C_j}$ is supermodular and $\tau(\emptyset)=0$. If $F\in\cA$ and $\tau(F)>0$, then $C_j\subseteq F\subseteq C_k$ for some $j,k$, and the antichain property forces $F=C_j=C_k$. Thus $L^{(s)}=J_s(\cA,\tau)$ and $\tau(E)=m$. Theorem~\ref{thm:3} gives
		\[
		s\omega(L)\le\omega(L^{(s)})\le\reg L^{(s)}\le sm.
		\]
		
		It remains to prove the bound $\reg L^{(s)}\le sD$ without an asymptotic argument. Let $E=\operatorname{supp}(L)$, write $L=I_\Gamma\subseteq T$, and put $d=\dim(T/L)$, so $D=d+1$. Then $L^{(s)}=\bigcap_{F\in\mathcal F(\Gamma)}P_{E\setminus F}^{\,s}$. Let $\mathbf a\in\ZZ^E$ be a normalized multidegree such that $\widetilde H_q(\Delta_{\mathbf a}(L^{(s)});\kk)\ne0$. Put $N=N(\mathbf a)$, $V=E\setminus N$, and let $\mathbf b=\mathbf a|_V$. The same localization argument as in Proposition~\ref{prop:1} gives
		\[
		\Delta_{\mathbf a}(L^{(s)})=
		\left\langle F\setminus N:
		F\in\mathcal F(\Gamma),\ N\subseteq F,\ b(E\setminus F)\le s-1
		\right\rangle.
		\]
		Assume first that $q\ge0$. If $b_i\ge s$ for some $i\in V$, then every simplex in the displayed generating family contains $i$; otherwise $b(E\setminus F)\ge b_i\ge s$. Hence the degree complex is a cone, contradicting the nonvanishing of its reduced homology. Thus $b_i\le s-1$ for every $i\in V$.
		
		Choose a facet $F$ occurring in the generating family. Since $N\subseteq F$ and $|F|\le d$,
		\[
		\begin{aligned}
			b(V)
			&=b(E\setminus F)+b(F\setminus N)\\
			&\le(s-1)+(s-1)|F\setminus N|\\
			&\le(s-1)(d-|N|+1).
		\end{aligned}
		\]
		Every generating simplex has at most $d-|N|$ vertices, so $q+2\le d-|N|+1$. By the regularity contribution formula above, the contribution of $\mathbf a$ to the regularity is at most
		\[
		b(V)+q+2\le s(d-|N|+1)\le s(d+1)=sD.
		\]
		If $q=-1$, then $\Delta_{\mathbf a}(L^{(s)})=\{\emptyset\}$. Hence some facet $F\in\mathcal F(\Gamma)$ satisfies $F=N$ and $b(E\setminus F)\le s-1$. Thus $V=E\setminus F$, and the regularity contribution formula above gives $b(V)+1\le s\le sD$. Therefore $\reg L^{(s)}\le sD$. Combining the two bounds gives $\reg L^{(s)}\le s\min\{m,D\}$. The remaining inequalities have already been proved, and the assertion for $s=1$ follows immediately.
	\end{proof}
	
	The two terms in the minimum are independently sharp. For 
	\[
	L=(x_1,x_2)\cap(x_3,x_4)=(x_1x_3,x_1x_4,x_2x_3,x_2x_4),
	\]
	we have $m=2$ and $D=3$. Since $L$ is matroidal and generated in degree $2$, Theorem~\ref{thm:1} gives $\reg L^{(s)}=2s=sm$. On the other hand, if $L$ is the squarefree Veronese ideal of degree $2$ in $n\ge3$ variables, equivalently the basis ideal of the uniform matroid $U_{2,n}$, then $m=n$ and $D=2$, while Theorem~\ref{thm:1} gives $\reg L^{(s)}=2s=sD$. Thus either $m$ or $D$ can determine the bound in Corollary~\ref{cor:2}.
	
	\begin{corollary}\label{cor:3}
		Let $L\subsetneq S$ be a nonzero squarefree monomial ideal. Put $m=|\operatorname{Min}(S/L)|$ and $D=\lambda(L)-\operatorname{ht}L+1$. If $\omega(L)=\min\{m,D\}$, then every $L^{(s)}$ has a degree resolution. More precisely, for every $s\ge1$,
		\[
		\omega(L^{(s)})=\reg L^{(s)}=s\omega(L).
		\]
		Moreover, if $L=I_\Gamma$ for a simplicial complex $\Gamma$ on $[n]$ and $\dim(S/L)=\omega(L)-1$, equivalently $\dim\Gamma=\omega(L)-2$, then the same conclusion holds.
	\end{corollary}
	
	\begin{proof}
		Set $e=\omega(L)=\min\{m,D\}$. The supermodular representation constructed in the proof of Corollary~\ref{cor:2}, together with Theorem~\ref{thm:3}, shows that $u^s\in \mathcal{G}(L^{(s)})$ for every $u\in \mathcal{G}(L)$. Choose $u\in \mathcal{G}(L)$ with $\deg u=e$. Then $se\le\omega(L^{(s)})$. On the other hand, Corollary~\ref{cor:2} gives $\omega(L^{(s)})\le\reg L^{(s)}\le s\min\{m,D\}=se$. Thus equality holds throughout. In particular,
		\begin{equation*}
			\omega(L^{(s)})=\reg L^{(s)}=se.
		\end{equation*}
		Hence every $L^{(s)}$ has a degree resolution.
		
		For the final assertion, put $e=\omega(L)$ and let $T$ be the polynomial ring on the variables occurring in $L$. Proposition~\ref{prop:2} with $s=1$ gives $e\le\dim(T/L)+1=D$. Deleting variables that do not occur in $L$ cannot increase the quotient dimension, so $D=\dim(T/L)+1\le\dim(S/L)+1=e$. Hence $e=D$. Corollary~\ref{cor:2} with $s=1$ now gives $e=\omega(L)\le\min\{m,D\}\le D=e$. Therefore $e=\min\{m,D\}$, and the first part applies.
	\end{proof}
	
	\begin{remark}\label{rem:3}
		When $L$ is unmixed, the regularity bound also follows directly from Theorem~\ref{thm:3}. Let $E$ be the support of $L$, let $h=\operatorname{ht}L$, and let $\cA$ be the lower ideal generated by the supports of the minimal primes. Every maximal member of $\cA$ has cardinality $h$. Define $\tau_h(F)=\max\{|F|-h+1,0\}$. This function is supermodular. On $\cA$, it is positive exactly on the supports of the minimal primes. Each of these supports has cardinality $h$, and $\tau_h$ takes the value $1$ on it. Thus $L^{(s)}=J_s(\cA,\tau_h)$, and Theorem~\ref{thm:3} gives
		\[
		\reg L^{(s)}\le s\tau_h(E)
		=s\bigl(\lambda(L)-h+1\bigr)
		=s\bigl(\dim(T/L)+1\bigr).
		\]
		For a general squarefree ideal the minimal primes need not have the same height, so this particular common supermodular representation need not exist. Corollary~\ref{cor:2} gives the same $sD$ bound without any unmixedness assumption by a direct degree complex argument.
	\end{remark}
	
	The final assertion of Corollary~\ref{cor:3} applies directly to skeleton and pure skeleton ideals. For a simplicial complex $\Delta$, let $\Delta^i=\{F\in\Delta:\dim F\le i\}$ and let $\Delta^{[i]}$ be the subcomplex generated by the $i$-dimensional faces.
	
	\begin{corollary}\label{cor:4}
		Let $\Delta$ be a simplicial complex on $[n]$ and let $-1\le i<\dim\Delta$. Then, for every $s\ge1$,
		\[
		\omega\bigl(I_{\Delta^i}^{(s)}\bigr)
		=\reg I_{\Delta^i}^{(s)}=s(i+2)
		\]
		and
		\[
		\omega\bigl(I_{\Delta^{[i]}}^{(s)}\bigr)
		=\reg I_{\Delta^{[i]}}^{(s)}=s(i+2).
		\]
		In particular, every symbolic power of either skeleton ideal has a degree resolution.
	\end{corollary}
	
	\begin{proof}
		Let $\Gamma$ be either $\Delta^i$ or $\Delta^{[i]}$, and put $J=I_\Gamma$. We first determine $\omega(J)$.
		
		Because $i<\dim\Delta$, there is a face $B\in\Delta$ with $|B|=i+2$. The set $B$ is not a face of $\Gamma$, since $\Gamma$ has dimension $i$. Every proper subset $A\subsetneq B$ does belong to $\Gamma$. For the skeleton this is immediate. For the pure skeleton, choose an $(i+1)$-element set $C$ with $A\subseteq C\subsetneq B$. Then $C$ is an $i$-face of $\Delta$, so $A$ belongs to the subcomplex generated by the $i$-faces. Therefore $B$ is a minimal nonface and $\omega(J)\ge i+2$. Conversely, let $C$ be any minimal nonface of $\Gamma$. If $|C|\ge i+3$, then an $(i+2)$-element subset $B'\subsetneq C$ is not a face because $\dim\Gamma=i$. This contradicts the minimality of $C$. Therefore every minimal nonface has size at most $i+2$, and $\omega(J)=i+2$. Since $\dim\Gamma=i$, we have $\dim(S/J)=i+1=\omega(J)-1$. Corollary~\ref{cor:3} gives both equalities.
	\end{proof}
	
	\begin{remark}
		For $s=1$, the formula for the skeleton agrees with \cite[Corollary~3.7]{Namiq}. Corollary~\ref{cor:4} extends the same equality to every symbolic power, so both the skeleton and pure skeleton ideals have degree resolutions in all symbolic powers.
	\end{remark}
	
	\section{Polymatroidal ideals}\label{sec:3}
	
	We apply the construction above to polymatroidal ideals. Let $E=[n]$, and for $\mathbf u=(u_i)_{i\in E}\in\ZZ_{\ge0}^E$ write $|\mathbf u|=\sum_{i\in E}u_i$ and $u(F)=\sum_{i\in F}u_i$ for $F\subseteq E$. We use the componentwise partial order, where $\mathbf v\leq\mathbf u$ means $v_i\leq u_i$ for every $i\in E$.
	
	A \emph{discrete polymatroid} on $E$ is a nonempty finite set $P\subseteq\ZZ_{\ge0}^E$ satisfying the following two conditions \cite{HerzogHibi}.
	\begin{enumerate}[label=\textup{(\roman*)}]
		\item If $\mathbf u\in P$ and $\mathbf v\in\ZZ_{\ge0}^E$ satisfies $\mathbf v\leq\mathbf u$, then $\mathbf v\in P$.
		\item If $\mathbf u,\mathbf v\in P$ and $|\mathbf u|<|\mathbf v|$, then there exists $i\in E$ such that $u_i<v_i$ and $\mathbf u+\mathbf e_i\in P$.
	\end{enumerate}
	Condition \textup{(ii)} is the augmentation axiom. Together with downward closure, it implies that all maximal elements of $P$ have the same modulus. If $d=\max\{|\mathbf u|:\mathbf u\in P\}$, then the maximal elements of $P$, called the \emph{bases} of $P$, are precisely the vectors $\mathbf u\in P$ with $|\mathbf u|=d$.
	
	The bases satisfy the corresponding exchange property. If $\mathbf u$ and $\mathbf v$ are bases and $u_i>v_i$ for some $i\in E$, then there exists $j\in E$ with $u_j<v_j$ such that $\mathbf u-\mathbf e_i+\mathbf e_j$ is again a base of $P$. Writing $x^{\mathbf u}=\prod_{i\in E}x_i^{u_i}$, the basis ideal
	\[
	I(P)=\bigl(x^{\mathbf u}:\mathbf u\text{ is a base of }P\bigr)\subseteq S
	\]
	is polymatroidal in the sense used in the Introduction and is generated in degree $d$. Conversely, the exponent vectors of the minimal generators of any polymatroidal ideal form the bases of a discrete polymatroid \cite{HerzogHibi}.
	
	The rank function of $P$ is $\rho(F)=\max_{\mathbf u\in P}u(F)$ for $F\subseteq E$. It satisfies $\rho(\emptyset)=0$ and $\rho(E)=d$, is nondecreasing, and is submodular. In other words,
	\[
	\rho(F)\leq\rho(G)\qquad\text{whenever }F\subseteq G\subseteq E,
	\]
	and
	\[
	\rho(F)+\rho(G)\geq\rho(F\cap G)+\rho(F\cup G)
	\qquad\text{for all }F,G\subseteq E.
	\]
	
	Define the complementary rank function
	\begin{equation*}
		\tau(F)=d-\rho(E\setminus F),
		\qquad F\subseteq E.
	\end{equation*}
	Then $\tau(\emptyset)=0$ and $\tau(E)=d$. Moreover, $\tau$ is supermodular. Indeed, submodularity of $\rho$ applied to $E\setminus A$ and $E\setminus B$ gives
	\[
	\rho(E\setminus A)+\rho(E\setminus B)
	\ge
	\rho(E\setminus(A\cap B))+
	\rho(E\setminus(A\cup B)),
	\]
	which is equivalent to
	\[
	\tau(A)+\tau(B)
	\le\tau(A\cap B)+\tau(A\cup B).
	\]
	
	Herzog and Vladoiu \cite[Theorem~8 and the proof of Theorem~12]{HerzogVladoiu} prove the following intersection representation associated with the rank function.
	\begin{equation}\label{eq:5}
		I(P)=\bigcap_{F\subseteq E}P_F^{\tau(F)},
		\qquad P_F=(x_i:i\in F),
	\end{equation}
	with factors of exponent zero omitted. The corresponding canonical primary decomposition is
	\[
	I(P)=\bigcap_{\substack{T\subseteq E\\ P_T\in\Ass(S/I(P))}}
	P_T^{\tau(T)}.
	\]
	The Minkowski sum of $s$ copies of $P$ has rank function $s\rho$, and its bases generate $I(P)^s$. Applying the intersection representation of Herzog and Vladoiu \eqref{eq:5} to this polymatroid gives
	\begin{equation}\label{eq:6}
		I(P)^s=\bigcap_{F\subseteq E}P_F^{s\tau(F)}.
	\end{equation}
	
	For a monomial ideal $I$, define
	\begin{equation*}
		\cA(I)=
		\bigl\{F\subseteq E:F\subseteq T\text{ for some }P_T\in\Ass(S/I)\bigr\}.
	\end{equation*}
	Thus $\cA(I)$ is the lower ideal generated by the supports of all associated primes, including embedded ones.
	
	\begin{proposition}\label{prop:3}
		Let $I=I(P)$. Then, for every $s\ge1$,
		\[
		I^{(s)}=\bigcap_{F\in\cA(I)}P_F^{s\tau(F)}.
		\]
	\end{proposition}
	
	\begin{proof}
		Fix an associated prime $P_T\in\Ass(S/I)$. By the definition of symbolic powers \eqref{eq:1}, we must compute $I^sS_{P_T}\cap S$. Using the intersection representation for $I^s$ in \eqref{eq:6} and localizing the finite intersection, we have
		\[
		I^sS_{P_T}
		=
		\bigcap_{F\subseteq E}P_F^{s\tau(F)}S_{P_T}.
		\]
		For a fixed $F\subseteq E$, suppose first that $F\nsubseteq T$. Choose $i\in F\setminus T$. Then $x_i\notin P_T$, so $x_i$ becomes a unit in $S_{P_T}$. Since $x_i\in P_F$, the localized ideal $P_FS_{P_T}$ is the whole ring. Thus $P_F^{s\tau(F)}$ localizes to the unit ideal and does not affect the intersection. If $F\subseteq T$, no variable generating $P_F$ is inverted. Since $P_F^{s\tau(F)}$ is $P_F$-primary and $P_F\subseteq P_T$, the multiplicative set $S\setminus P_T$ is disjoint from the radical $P_F$. Localization followed by contraction therefore recovers the same primary ideal: $P_F^{s\tau(F)}S_{P_T}\cap S=P_F^{s\tau(F)}$. Therefore
		\begin{equation*}
			I^sS_{P_T}\cap S
			=
			\bigcap_{F\subseteq T}P_F^{s\tau(F)}.
		\end{equation*}
		Intersecting the last display over all $P_T\in\Ass(S/I)$ gives the symbolic power. A factor indexed by $F$ occurs in the intersection if and only if $F\subseteq T$ for at least one associated prime $P_T$. By definition, this is equivalent to $F\in\cA(I)$. This proves Proposition~\ref{prop:3}. The same argument applies to minimal and embedded associated primes, so $\cA(I)$ must be generated by the supports of all associated primes.
	\end{proof}
	
	\begin{proof}[Proof of Theorem~\ref{thm:1}]
		Let $I=I(P)$. By Proposition~\ref{prop:3}, the symbolic powers have the common representation $I^{(s)}=J_s(\cA(I),\tau)$. The complementary rank function satisfies $\tau(E)=d$, while $J_1=I$ is generated in degree $d$. It follows that $\omega(J_1)=d=\tau(E)$. The equality case of Theorem~\ref{thm:3} gives
		\[
		\omega(I^{(s)})=\reg I^{(s)}=sd.
		\]
		Therefore every symbolic power $I^{(s)}$ has a degree resolution. Also, $I^s$ is polymatroidal and generated in degree $sd$ \cite{ConcaHerzog,HerzogHibi}. Since polymatroidal ideals have linear quotients, $I^s$ has an $sd$-linear resolution. Therefore $\reg I^s=sd=\reg I^{(s)}$, which completes the proof.
	\end{proof}
	
	\begin{corollary}\label{cor:5}
		Let $M$ be a matroid of rank $d$ on $[n]$, let $\mathcal B(M)$ be its set of bases, and let
		\[
		I_B(M)=(x_B:B\in\mathcal B(M))\subseteq S,
		\qquad x_B=\prod_{i\in B}x_i.
		\]
		Then every symbolic power $I_B(M)^{(s)}$ has a degree resolution and
		\[
		\omega(I_B(M)^{(s)})=\reg I_B(M)^{(s)}=sd
		\qquad(s\ge1).
		\]
	\end{corollary}
	
	Two further standard families give immediate consequences of Theorem~\ref{thm:1}.
	
	\begin{corollary}\label{cor:6}
		Let $\mathbf a=(a_1,\ldots,a_n)\in\ZZ_{\ge0}^n$ and let $d\ge1$ satisfy $a_1+\cdots+a_n\ge d$. Put
		\[
		I_{\mathbf a,d}=\bigl(x^{\mathbf u}:|\mathbf u|=d,\ 0\le u_i\le a_i\text{ for all }i\bigr)\subseteq S.
		\]
		Then every symbolic power $I_{\mathbf a,d}^{(s)}$ has a degree resolution and
		\[
		\omega(I_{\mathbf a,d}^{(s)})=\reg I_{\mathbf a,d}^{(s)}=\reg I_{\mathbf a,d}^s=sd
		\qquad(s\ge1).
		\]
	\end{corollary}
	
	\begin{proof}
		The ideal $I_{\mathbf a,d}$ is polymatroidal. Indeed, if $x^{\mathbf u}$ and $x^{\mathbf v}$ are generators and $u_i>v_i$, then $|\mathbf u|=|\mathbf v|$ gives an index $j$ with $u_j<v_j$. Hence $\mathbf u-\mathbf e_i+\mathbf e_j$ still has nonnegative coordinates, total degree $d$, and $u_j+1\le v_j\le a_j$. Thus the exchange property holds, and the assertion follows from Theorem~\ref{thm:1}.
	\end{proof}
	
	\begin{corollary}\label{cor:7}
		Let $I_1,\ldots,I_r\subseteq S$ be nonzero polymatroidal ideals generated in degrees $d_1,\ldots,d_r$, respectively, and put $J=I_1\cdots I_r$. Then every symbolic power $J^{(s)}$ has a degree resolution and
		\[
		\omega(J^{(s)})=\reg J^{(s)}=\reg J^s=s(d_1+\cdots+d_r)
		\qquad(s\ge1).
		\]
	\end{corollary}
	
	\begin{proof}
		Products of polymatroidal ideals are polymatroidal
		\cite[Theorem~5.3]{ConcaHerzog}. Moreover, $J$ is generated in degree $d_1+\cdots+d_r$. The assertion follows from Theorem~\ref{thm:1}.
	\end{proof}
	
	\begin{remark}
		The preceding corollaries illustrate both the scope and the limitations of Theorem~\ref{thm:1}.
		\begin{enumerate}[label=\textup{(\roman*)}]
			\item Minh and Trung computed the regularity of symbolic powers of Stanley--Reisner ideals of matroid complexes \cite{MinhTrung}, while Mantero and Nguyen studied symbolic powers of Stanley--Reisner and cover ideals of matroids \cite{ManteroNguyen}. Corollary~\ref{cor:5}, by contrast, concerns the basis ideal of a matroid.
			
			\item Corollary~\ref{cor:6} includes the ordinary Veronese ideals and the squarefree Veronese ideals. More generally, Theorem~\ref{thm:1} applies to other standard classes of polymatroidal ideals, including principal Borel ideals and transversal polymatroidal ideals; see \cite{FicarraMoradi,HerzogHibi}. Corollary~\ref{cor:7} also shows that the conclusion is preserved under finite products of polymatroidal ideals.
			
			\item Theorem~\ref{thm:1} does not extend to arbitrary weakly polymatroidal ideals. Indeed, there are weakly polymatroidal ideals generated in one degree whose ordinary powers do not have linear resolutions \cite[Remark~1.5]{MohammadiMoradi}, so the equality $\reg I^s=sd$ may fail. Nevertheless, the squarefree results of Section~\ref{sec:2} apply to every squarefree weakly polymatroidal ideal. In particular, if $L$ is squarefree and weakly polymatroidal, then Corollary~\ref{cor:2} gives
			\[
			s\omega(L)\le\omega(L^{(s)})\le\reg L^{(s)}
			\le s\min\{m,D\},
			\]
			and if $\omega(L)=\min\{m,D\}$, then Corollary~\ref{cor:3} shows that every $L^{(s)}$ has a degree resolution.
		\end{enumerate}
	\end{remark}
	
	We conclude the section by turning to componentwise linearity. Theorem~\ref{thm:1} shows that every symbolic power of a polymatroidal ideal has a degree resolution, but this does not imply componentwise linearity. For a homogeneous ideal $J$ and an integer $q$, let $J_{\langle q\rangle}$ denote the ideal generated by the homogeneous elements of $J$ of degree $q$. The counterexamples below are basis ideals of simple graphic matroids.
	
	\begin{figure}[ht]
		\centering
		\begin{tikzpicture}[
			scale=.82,
			vertex/.style={circle,draw=black,fill=black,minimum size=5.5pt,inner sep=0pt},
			edge/.style={line width=0.7pt},
			lab/.style={fill=white,inner sep=1.1pt,font=\scriptsize}
			]
			\node[vertex] (v1) at (0,1.35) {};
			\node[vertex] (v2) at (-1.35,-.85) {};
			\node[vertex] (v3) at (1.35,-.85) {};
			\node[vertex] (v4) at (0,-.15) {};
			
			\draw[edge] (v1)-- node[lab,pos=.52,left] {$1$} (v2);
			\draw[edge] (v1)-- node[lab,pos=.52,right] {$2$} (v3);
			\draw[edge] (v1)-- node[lab,pos=.52,left] {$3$} (v4);
			\draw[edge] (v2)-- node[lab,pos=.50,below] {$4$} (v3);
			\draw[edge] (v2)-- node[lab,pos=.54,above] {$5$} (v4);
			\draw[edge] (v3)-- node[lab,pos=.54,above] {$6$} (v4);
		\end{tikzpicture}
		\caption{The graph $K_4$ used in the rank $3$ construction}
		\label{fig:1}
	\end{figure}
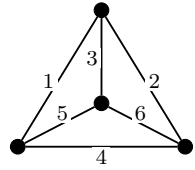
	
	\begin{proposition}\label{prop:4}
		For every integer $d\ge3$, there exists a simple graphic matroid $M_d$ of rank $d$ such that $I_B(M_d)^{(2)}$ is not componentwise linear. More precisely,
		\[
		\reg\left(\left(I_B(M_d)^{(2)}\right)_{\langle2d-2\rangle}\right)=2d-1.
		\]
	\end{proposition}
	
	\begin{proof}
		We first consider rank $3$. Let $M_3=M(K_4)$ and label the six edges of $K_4$, as in Figure~\ref{fig:1}, by $1=12$, $2=13$, $3=14$, $4=23$, $5=24$, and $6=34$. Put $I_3=I_B(M_3)\subseteq S_3=\kk[x_1,\ldots,x_6]$. In any matroid, a cocircuit is a minimal set that meets every basis. Therefore the cocircuits are the minimal vertex covers of the basis hypergraph. For a cycle matroid, the cocircuits are the bonds of the graph. In $K_4$ there are four bonds with $3$ edges corresponding to cuts that isolate one vertex and three bonds with $4$ edges corresponding to the three partitions of the vertex set into two pairs. The cocircuits are therefore
		\begin{align*}
			&\{1,2,3\},\ \{1,4,5\},\ \{2,4,6\},\ \{3,5,6\},\\
			&\{1,2,5,6\},\ \{1,3,4,6\},\ \{2,3,4,5\}.
		\end{align*}
		Since $I_3$ is squarefree, it is radical, and its minimal primes are the monomial primes corresponding to these seven minimal vertex covers. Thus $I_3^{(2)}$ is the intersection of their squares.
		
		We determine the degree $4$ component. Let $x^{\mathbf a}\in I_3^{(2)}$ have total degree $4$. The four $3$-element cocircuits give
		\begin{align*}
			a_1+a_2+a_3&\ge2,&
			a_1+a_4+a_5&\ge2,\\
			a_2+a_4+a_6&\ge2,&
			a_3+a_5+a_6&\ge2.
		\end{align*}
		Each variable occurs in exactly two of these sums. Adding the four inequalities therefore gives $2|\mathbf a|\ge8$. But $|\mathbf a|=4$, so equality holds. Because each of the four left sides is an integer at least $2$, their sum can equal $8$ only if all four inequalities are equalities.
		
		Subtracting the first two equations and, separately, the third and fourth equations gives
		\[
		a_2+a_3=a_4+a_5,
		\qquad
		a_2+a_4=a_3+a_5.
		\]
		Subtracting the two relations gives $a_3=a_4$, and then $a_2=a_5$. Comparing the first and fourth equations gives $a_1=a_6$. We obtain
		\[
		a_1=a_6=r,
		\qquad
		a_2=a_5=t,
		\qquad
		a_3=a_4=2-r-t,
		\]
		where $r,t\in\ZZ_{\ge0}$ and $r+t\le2$. The six possibilities are
		\[
		(0,0),\ (1,0),\ (0,1),\ (2,0),\ (0,2),\ (1,1).
		\]
		The choices $(2,0)$, $(0,2)$, and $(0,0)$ give, respectively,
		\[
		x_1^2x_6^2,\qquad x_2^2x_5^2,\qquad x_3^2x_4^2.
		\]
		They fail the inequalities associated with the cocircuits with $4$ edges $\{2,3,4,5\}$, $\{1,3,4,6\}$, and $\{1,2,5,6\}$, respectively. The remaining choices $(1,0)$, $(0,1)$, and $(1,1)$ give, respectively, $ x_1x_3x_4x_6$, $x_2x_3x_4x_5$, and $x_1x_2x_5x_6$. Each contains at least two variables from every cocircuit with $4$ edges, so all three satisfy the remaining inequalities. Consequently
		\begin{equation*}
			\left(I_3^{(2)}\right)_{\langle4\rangle}
			=(m_1,m_2,m_3),
		\end{equation*}
		where
		\[
		m_1=x_2x_3x_4x_5,
		\qquad
		m_2=x_1x_3x_4x_6,
		\qquad
		m_3=x_1x_2x_5x_6.
		\]
		
		To compute the regularity of this component, set $A=x_1x_6$, $B=x_2x_5$, and $C=x_3x_4$. Then $m_1=BC$, $m_2=AC$, and $m_3=AB$, and $A,B,C$ are pairwise coprime. Two syzygies are
		\[
		(A,-B,0)
		\qquad\text{and}\qquad
		(A,0,-C).
		\]
		They have total degree $6$. We verify that they generate all first syzygies. Suppose $fBC+gAC+hAB=0$. Reducing modulo $(A)$ gives $fBC\in(A)$. Since $\gcd(A,BC)=1$, $A\mid f$. Write $f=Aq$. After dividing the relation by $A$, we have $qBC+gC+hB=0$. Reducing modulo $(B)$ gives $gC\in(B)$, and $\gcd(B,C)=1$, so $B\mid g$. Write $g=-Bp$. Then $B\bigl(C(q-p)+h\bigr)=0$, so $h=C(p-q)$. Therefore $(f,g,h)=p(A,-B,0)+(q-p)(A,0,-C)$. These two syzygies are $S_3$-linearly independent and generate the first syzygy module. Since every entry of the differential has positive degree, this resolution is minimal. Therefore this ideal has the minimal free resolution
		\[
		0\longrightarrow S_3(-6)^2
		\longrightarrow S_3(-4)^3
		\longrightarrow
		\left(I_3^{(2)}\right)_{\langle4\rangle}
		\longrightarrow0.
		\]
		Hence
		\begin{equation*}
			\reg\left(\left(I_3^{(2)}\right)_{\langle4\rangle}\right)
			=\max\{4,6-1\}=5.
		\end{equation*}
		The ideal is generated in degree $4$ but does not have a $4$-linear resolution. Therefore $I_3^{(2)}$ is not componentwise linear.
		
		For higher ranks, let $d\ge3$, set $k=d-3$, and let $M_d=M_3\oplus U_{k,k}$. All $k$ elements $y_1,\ldots,y_k$ of $U_{k,k}$ are coloops. The matroid $M_d$ is the cycle matroid of the graph obtained by attaching a path of $k$ bridge edges to $K_4$. Thus $M_d$ is again a simple graphic matroid, and its rank is $d$. Put $S_d=S_3[y_1,\ldots,y_k]$ and $Y=y_1\cdots y_k$, with $Y=1$ when $k=0$. Every basis contains all coloops, so $I_B(M_d)=YI_3$. The minimal vertex covers of the new basis hypergraph are the seven cocircuits inherited from $M_3$ together with the singletons $\{y_1\},\ldots,\{y_k\}$. After extending $I_3^{(2)}$ to $S_d$, we get
		\[
		I_B(M_d)^{(2)}
		=I_3^{(2)}\cap(y_1)^2\cap\cdots\cap(y_k)^2.
		\]
		The ideal $I_3^{(2)}$ involves only the variables $x_1,\ldots,x_6$, whereas each $(y_j)^2$ involves a new variable. The intersection therefore equals the product
		\[
		I_B(M_d)^{(2)}=Y^2I_3^{(2)}.
		\]
		The four $3$-element cocircuit inequalities show that $I_3^{(2)}$ has no monomials of degree below $4$. Since $2d-2=4+2k$,
		\[
		\left(I_B(M_d)^{(2)}\right)_{\langle2d-2\rangle}
		=Y^2\left(I_3^{(2)}\right)_{\langle4\rangle}.
		\]
		Multiplication by $Y^2$ shifts every degree in the minimal resolution by $2k$. Using the regularity computation above,
		\[
		\reg\left(\left(I_B(M_d)^{(2)}\right)_{\langle2d-2\rangle}\right)
		=5+2k=2d-1.
		\]
		This component is generated in degree $2d-2$, while its regularity is $2d-1$. Thus it does not have a linear resolution. It follows that $I_B(M_d)^{(2)}$ is not componentwise linear.
	\end{proof}
	
	\begin{corollary}\label{cor:8}
		Let $d\ge1$. The following conditions are equivalent.
		\begin{enumerate}[label=\textup{(\roman*)}]
			\item $d\le2$.
			\item every symbolic power of every polymatroidal ideal generated in degree $d$ is componentwise linear.
		\end{enumerate}
	\end{corollary}
	
	\begin{proof}
		If $d=1$, every polymatroidal ideal generated in degree $1$ is generated by a subset of the variables, hence is a monomial prime. Its symbolic and ordinary powers agree and have linear resolutions. If $d=2$, let $I$ be a polymatroidal ideal generated in degree $2$. If $I$ is squarefree, assertion (ii) follows from \cite[Proposition~6.1]{FicarraMoradi}. Otherwise, \cite[Theorem~5.8(a)]{FicarraMoradi} gives $I^{(s)}=I^s$ for all $s\ge1$, so these powers have linear resolutions. This proves (i)$\Rightarrow$(ii).
		
		Conversely, if $d\ge3$, Proposition~\ref{prop:4} gives a matroidal ideal generated in degree $d$, hence a polymatroidal ideal, whose second symbolic power is not componentwise linear. Then (ii) fails, which proves (ii)$\Rightarrow$(i).
	\end{proof}
	
	\begin{remark}
		Proposition~\ref{prop:4} disproves Conjecture~B of Ficarra and Moradi \cite{FicarraMoradi} in every generating degree $d\ge3$, even for basis ideals of simple graphic matroids. Since the degree $2$ case is positive, rank $3$ is the smallest possible rank for a matroidal counterexample. The same $K_4$ example gives a negative answer to the componentwise linearity part of Question~D in \cite{FicarraMoradi}.
	\end{remark}


\begin{thebibliography}{99}
		\small
		
		\bibitem{AhmedNamiqCoChordal}
		C.~Ahmed and M.~R. Namiq,
		\emph{Regularity of symbolic powers of co-chordal edge ideals},
		arXiv:2608.20542 (2026).
		\url{https://arxiv.org/abs/2608.20542}.
		
		\bibitem{Bjorner}
		A.~Bj\"orner,
		\emph{Topological methods},
		in R.~L. Graham, M.~Gr\"otschel, and L.~Lov\'asz (eds.),
		\emph{Handbook of Combinatorics}, Vols.~1--2,
		Elsevier, Amsterdam, 1995, pp.~1819--1872.
		
		\bibitem{CHHKTT}
		G.~Caviglia, H.~T. H\`a, J.~Herzog, M.~Kummini, N.~Terai, and N.~V. Trung,
		\emph{Depth and regularity modulo a principal ideal},
		J. Algebraic Combin. \textbf{49} (2019), no.~1, 1--20.
		\doi{10.1007/s10801-018-0811-9}.
		
		\bibitem{ChoiKimPark}
		I.~Choi, J.~Kim, and B.~Park,
		\emph{Collapsibility of non-cover complexes of graphs},
		Electron. J. Combin. \textbf{27} (2020), no.~1, Paper~P1.20.
		\doi{10.37236/8684}.
		
		\bibitem{ConcaHerzog}
		A.~Conca and J.~Herzog,
		\emph{Castelnuovo--Mumford regularity of products of ideals},
		Collect. Math. \textbf{54} (2003), no.~2, 137--152.
		
		\bibitem{DungHienNguyenTrung}
		L.~X. Dung, T.~T. Hien, H.~D. Nguyen, and T.~N. Trung,
		\emph{Regularity and Koszul property of symbolic powers of monomial ideals},
		Math. Z. \textbf{298} (2021), no.~3--4, 1487--1522.
		\doi{10.1007/s00209-020-02657-8}.
		
		\bibitem{FicarraMoradi}
		A.~Ficarra and S.~Moradi,
		\emph{Symbolic powers of polymatroidal ideals},
		J. Pure Appl. Algebra \textbf{229} (2025), no.~10, Article~108082.
		\doi{10.1016/j.jpaa.2025.108082}.
		
		\bibitem{FicarraMoradiRomer}
		A.~Ficarra, S.~Moradi, and T.~R\"omer,
		\emph{Componentwise linear symbolic powers of edge ideals and Minh's conjecture},
		J. Algebraic Combin. \textbf{62} (2025), no.~4, Article~56.
		\doi{10.1007/s10801-025-01485-5}.
		
		\bibitem{HaNgo}
		H.~T. H\`a and N.~V. Trung,
		\emph{Membership criteria and containments of powers of monomial ideals},
		Acta Math. Vietnam. \textbf{44} (2019), no.~1, 117--139.
		\doi{10.1007/s40306-018-00325-y}.
		
		
		\bibitem{HerzogHibi}
		J.~Herzog and T.~Hibi,
		\emph{Discrete polymatroids},
		J. Algebraic Combin. \textbf{16} (2002), 239--268.
		\doi{10.1023/A:1021852421716}.
		
		\bibitem{HerzogVladoiu}
		J.~Herzog and M.~Vladoiu,
		\emph{Monomial ideals with primary components given by powers of monomial prime ideals},
		Electron. J. Combin. \textbf{21} (2014), no.~1, Paper~P1.69.
		\doi{10.37236/3813}.
		
		\bibitem{KimKim}
		J.~Kim and M.~Kim,
		\emph{Domination numbers and noncover complexes of hypergraphs},
		J. Combin. Theory Ser. A \textbf{180} (2021), Article~105408.
		\doi{10.1016/j.jcta.2021.105408}.
		
		\bibitem{ManteroNguyen}
		P.~Mantero and V.~Nguyen,
		\emph{The structure of symbolic powers of matroids},
		Trans. Amer. Math. Soc. \textbf{379} (2026), no.~6, 3855--3894.
		\doi{10.1090/tran/9421}.
		
		\bibitem{MinhTrung}
		N.~C. Minh and T.~N. Trung,
		\emph{Regularity of symbolic powers and arboricity of matroids},
		Forum Math. \textbf{31} (2019), no.~2, 465--477.
		\doi{10.1515/forum-2017-0243}.
		
		\bibitem{MohammadiMoradi}
		F.~Mohammadi and S.~Moradi,
		\emph{Weakly polymatroidal ideals with applications to vertex cover ideals},
		Osaka J. Math. \textbf{47} (2010), no.~3, 627--636.
		\doi{10.18910/12673}.
		
		\bibitem{Namiq}
		M.~R. Namiq,
		\emph{The graded Betti numbers of the skeletons of simplicial complexes},
		Indag. Math. (N.S.) (2026), in press.
		\doi{10.1016/j.indag.2026.03.006}.
		
		\bibitem{Takayama}
		Y.~Takayama,
		\emph{Combinatorial characterizations of generalized Cohen--Macaulay monomial ideals},
		Bull. Math. Soc. Sci. Math. Roumanie (N.S.) \textbf{48(96)} (2005), no.~3, 327--344.
		
	\end{thebibliography}
\end{document}